\documentclass[11pt,a4paper]{article}

\usepackage[utf8]{inputenc}
\usepackage[T1]{fontenc}
\usepackage{amsmath}
\usepackage{amssymb} 
\usepackage{amsthm}
\usepackage{enumerate}
\usepackage{graphicx}
\usepackage[margin=1in]{geometry}

\usepackage{authblk}

\usepackage{hyperref}

\newtheorem{theorem}{Theorem}

\newtheorem{lemma}[theorem]{Lemma}

\newtheorem{remark}[theorem]{Remark}

\author[1]{Ulrich Abel}
\author[2,3]{Naim L. Braha} 

\affil[1]{Technische Hochschule Mittelhessen, Department Mathematik, Naturwissenschaften und Datenverarbeitung, Wilhelm-Leuschner-Straße 13, 61169 Friedberg, Germany. \textit{E-mail: Ulrich.Abel@mnd.thm.de}}
\affil[2]{Department of Mathematics and Computer Sciences, University of Prishtina, Avenue George Bush, No-5, 10000 Prishtina, Kosova}
\affil[3]{Ilirias Research Institute (\url{http://ilirias.com}), Janina No-2, 70000 Ferizaj, Kosova. \textit{E-mail: nbraha@yahoo.com}}

\begin{document}

\title{\textbf{Complete asymptotic expansion for a Durrmeyer variant of
operators based on Hermite polynomials}}
\date{}
\maketitle

\begin{abstract}
In this paper, we study a Durrmeyer variant of the positive linear operators
based on two-variable Hermite polynomials recently introduced by G. Krech
(2016). Our main objective is to establish a complete asymptotic expansion
for these operators as $n \to \infty$ for locally integrable functions of
polynomial growth. The coefficients of the expansion are explicitly
expressed in terms of the derivatives of the function. As a corollary, a
Voronovskaja-type formula is obtained.
\end{abstract}


\vspace{0.5cm}


\section{Introduction}

\label{Intro}

In 2016, Gra\.{z}yna Krech \cite{Krech2016} introduced an interesting class
of positive linear operators $G_{n,\alpha }$, $n\in \mathbb{N}$, $\alpha
\geq 0$, which are defined for $x\geq 0$ and sufficiently regular functions $%
f$ by 
\begin{equation}
\left( G_{n,\alpha }f\right) \left( x\right) =e^{-\left( nx+\alpha
x^{2}\right) }\sum_{\nu =0}^{\infty }\frac{x^{\nu }}{\nu !}H_{\nu }\left(
n,\alpha \right) f\left( \frac{\nu }{n}\right) ,  \label{eq:Krech-op}
\end{equation}%
where $H_{\nu }(n,\alpha )$ denotes the two-variable Hermite polynomial
defined by 
\begin{equation*}
H_{\nu }\left( n,\alpha \right) =\nu !\sum_{j=0}^{\left\lfloor \nu
/2\right\rfloor }\frac{n^{\nu -2j}\alpha ^{j}}{\left( \nu -2j\right) !j!}.
\end{equation*}%
The exponential generating function of these polynomials is given by 
\begin{equation*}
\sum_{\nu =0}^{\infty }\frac{x^{\nu }}{\nu !}H_{\nu }\left( n,\alpha \right)
=\exp \left( nx+\alpha x^{2}\right) .
\end{equation*}%
From this it immediately follows that, for $j=0,1,2,\ldots $, 
\begin{equation}
\sum_{\nu =0}^{\infty }\frac{x^{\nu }}{\nu !}H_{\nu +j}\left( n,\alpha
\right) =\left( \frac{d}{dx}\right) ^{j}\exp \left( nx+\alpha x^{2}\right) .
\label{GF}
\end{equation}%
As $\alpha \geq 0$, the operators (\ref{eq:Krech-op}) are linear and
positive. They generalize the classical Sz\'{a}sz--Mirakyan operators, which
are recovered in the special case $\alpha =0$. Approximation properties and
Voronovskaja-type theorems for these discrete operators and related forms
have been studied recently in \cite{AbelBraha2025, Gupta2026}.

In order to approximate Lebesgue integrable functions, it is natural to
consider integral modifications of discrete operators. The main goal of this
paper is to introduce and investigate a Durrmeyer-type modification of the
operators (\ref{eq:Krech-op}). We derive a complete asymptotic expansion for
this variant, which generalizes the classical Sz\'{a}sz--Mirakyan--Durrmeyer
operators. The corresponding result for the operators (\ref{eq:Krech-op})
can be found in \cite[Theorem 2.1]{AbelBraha2025}.

The paper is organized as follows: In Section 2, we formally define the
Durrmeyer variant $\bar{G}_{n,\alpha }$ and state our main result regarding
its complete asymptotic expansion. Section 3 is devoted to auxiliary lemmas,
central moments, and the proof of the main theorem.

\section{Durrmeyer variant}

For $n\in \mathbb{N}$, $\alpha \geq 0$, and $x\geq 0$, we define 
\begin{equation}
\left( \bar{G}_{n,\alpha }f\right) \left( x\right) =e^{-\left( nx+\alpha
x^{2}\right) }\sum_{\nu =0}^{\infty }\frac{x^{\nu }}{\nu !}H_{\nu }\left(
n,\alpha \right) \Phi _{n,\nu }\left( f\right) ,
\label{def-G-n-alpha-Durrmeyer}
\end{equation}%
where the point evaluation $f\left( \frac{\nu }{n}\right) $ of the operator (%
\ref{eq:Krech-op}) is replaced with the integral 
\begin{equation*}
\Phi _{n,\nu }\left( f\right) :=n\int_{0}^{\infty }e^{-nt}\frac{\left(
nt\right) ^{\nu }}{\nu !}f\left( t\right) dt.
\end{equation*}%
In the special case $\alpha =0$,\ one has $H_{\nu }\left( n,0\right) =n^{\nu
}$ and the operators $\left( \ref{def-G-n-alpha-Durrmeyer}\right) $ reduce
to the classical Sz\'{a}sz--Mirakyan--Durrmeyer operators 
\begin{equation*}
\left( G_{n,0}f\right) \left( x\right) =e^{-nx}\sum_{\nu =0}^{\infty }\frac{%
\left( nx\right) ^{\nu }}{\nu !}n\int_{0}^{\infty }e^{-nt}\frac{\left(
nt\right) ^{\nu }}{\nu !}f\left( t\right) dt.
\end{equation*}

For $s\in \mathbb{N}$ and $x\in \left( 0,\infty \right) $, let $K\left[ s;x%
\right] $ be the class of all locally integrable functions $f$ on $\mathbb{R}%
_{0}^{+}$, which satisfy the growth condition $f\left( t\right) =O\left(
t^{s}\right) $ as $t\rightarrow +\infty $, and which are $s$ times
differentiable at $x$. The following theorem presents as our main result the
complete asymptotic expansion for the operators $G_{n,\alpha }$.

\begin{theorem}
\label{theorem-expansion}Let $q\in \mathbb{N}$ and $x\in \left( 0,\infty
\right) $. For each function $f\in K\left[ 2q;x\right] $, the operators $%
\bar{G}_{n,\alpha }$ possess the asymptotic expansion 
\begin{equation}
\left( \bar{G}_{n,\alpha }f\right) \left( x\right) =\sum_{k=0}^{q}\frac{%
c_{k,\alpha }\left( f,x\right) }{n^{k}}+o\left( n^{-q}\right) \qquad \left(
n\rightarrow \infty \right)  \label{complete-asympt-expansion-in-theorem}
\end{equation}%
with the coefficients 
\begin{equation}
c_{k,\alpha }\left( f,x\right) =\frac{1}{k!}\sum_{i=0}^{k}\binom{k}{i}\left(
2\alpha \right) ^{k-i}\left( x^{2k-i}f^{\left( k\right) }\left( x\right)
\right) ^{\left( i\right) }  \label{coeff-c}
\end{equation}%
$\left( k=0,1,2,\ldots \right) $.
\end{theorem}

\begin{remark}
Let $x\in \left( 0,\infty \right) $. For each function $f\in
\bigcap_{q=1}^{\infty }K\left[ 2q;x\right] $, the operators $\bar{G}%
_{n,\alpha }$ possess the complete asymptotic expansion 
\begin{equation*}
\left( \bar{G}_{n,\alpha }f\right) \left( x\right) \sim \sum_{k=0}^{\infty }%
\frac{1}{k!n^{k}}\sum_{i=0}^{k}\binom{k}{i}\left( 2\alpha \right)
^{k-i}\left( x^{2k-i}f^{\left( k\right) }\left( x\right) \right) ^{\left(
i\right) }\qquad \left( n\rightarrow \infty \right) .
\end{equation*}
\end{remark}

\begin{remark}
For the convenience of the reader, we list the explicit expressions for the
initial coefficients $c_{k,\alpha }\left( f,x\right) $ of the Durrmeyer
variant: 
\begin{align*}
c_{0,\alpha }\left( f,x\right) & =f(x) \\[0.5em]
c_{1,\alpha }\left( f,x\right) & =\left( 2\alpha x^{2}+1\right) f^{\prime
}(x)+xf^{\prime \prime }(x) \\[0.5em]
c_{2,\alpha }\left( f,x\right) & =\left( 2\alpha ^{2}x^{4}+6\alpha
x^{2}+1\right) f^{\prime \prime }(x)+\left( 2\alpha x^{3}+2x\right)
f^{(3)}(x)+\frac{1}{2}x^{2}f^{(4)}(x) \\[0.5em]
c_{3,\alpha }\left( f,x\right) & =\left( \frac{4}{3}\alpha
^{3}x^{6}+10\alpha ^{2}x^{4}+12\alpha x^{2}+1\right) f^{(3)}(x) \\
& \quad +\left( 2\alpha ^{2}x^{5}+8\alpha x^{3}+3x\right) f^{(4)}(x) \\
& \quad +\left( \alpha x^{4}+\frac{3}{2}x^{2}\right) f^{(5)}(x)+\frac{1}{6}%
x^{3}f^{(6)}(x)
\end{align*}
\end{remark}

In the particular case $q=1$, we obtain the following Voronovskaja-type
formula.

\begin{theorem}
Let $x\in \left( 0,\infty \right) $. For each function $f\in K\left[ 2;x%
\right] $, the operators $\bar{G}_{n,\alpha }$ possess the asymptotic
expansion 
\begin{equation*}
\lim_{n\rightarrow \infty }n\left( \left( \bar{G}_{n,\alpha }f\right) \left(
x\right) -f\left( x\right) \right) =\left( 2\alpha x^{2}+1\right) f^{\prime
}\left( x\right) +xf^{\prime \prime }\left( x\right) .
\end{equation*}
\end{theorem}

\section{Auxiliary results and proof of the main theorem}

Firstly, we study the moments of the operators $G_{n,\alpha }$. Throughout
the paper, let $e_{r}$ denote the monomials, given by $e_{r}\left( x\right)
=x^{r}$ $\left( r=0,1,2,\ldots \right) $. Furthermore, define $\psi
_{x}=e_{1}-xe_{0}$, for $x\in \mathbb{R}$.

\begin{lemma}
\label{lemma-moments}For $r=0,1,2,\ldots $, the moments of the operators $%
\left( \ref{eq:Krech-op}\right) $ have the explicit representation 
\begin{equation*}
\left( \bar{G}_{n,\alpha }e_{r}\right) \left( x\right) =\sum_{k=0}^{r}\frac{1%
}{n^{k}}c_{k,\alpha }\left( e_{r},x\right) ,
\end{equation*}%
where 
\begin{equation*}
c_{k,\alpha }\left( e_{r},x\right) =\frac{1}{k!}\sum_{i=0}^{k}\binom{k}{i}%
\left( x^{2k-i}e_{r}^{\left( k\right) }\left( x\right) \right) ^{\left(
i\right) }\left( 2\alpha \right) ^{k-i}.
\end{equation*}
\end{lemma}

\begin{remark}
Note that the exponent $k+r-2i$ of the variable $x$ in the expression $%
\left( x^{2k-i}e_{r}^{\left( k\right) }\left( x\right) \right) ^{\left(
i\right) }$ is nonnegative, since non-zero terms can occur only if $0\leq
i\leq \min \left\{ k,r\right\} $. The moments $G_{n,\alpha }e_{r}$ are
polynomials in $x$ and $\alpha $ of degree $2r$ and $r$, respectively. The
leading term of $\left( \bar{G}_{n,\alpha }e_{r}\right) \left( x\right) $ is
given by $\left( 2/n\right) ^{r}x^{2r}\alpha ^{r}$.
\end{remark}

\begin{proof}
Using 
\begin{equation*}
\Phi _{n,\nu }\left( e_{r}\right) =n^{-r}\int_{0}^{\infty }e^{-t}\frac{%
t^{\nu +r}}{\nu !}dt=\frac{\left( \nu +r\right) !}{n^{r}\nu !}=\frac{1}{n^{r}%
}\left. \left( \frac{d}{dz}\right) ^{r}z^{\nu +r}\right\vert _{z=1}
\end{equation*}%
we obtain\ 
\begin{eqnarray*}
\left( \bar{G}_{n,\alpha }e_{r}\right) \left( x\right) &=&n^{-r}e^{-\left(
nx+\alpha x^{2}\right) }\left. \left( \frac{d}{dz}\right) ^{r}\left(
z^{r}\sum_{\nu =0}^{\infty }\frac{\left( xz\right) ^{\nu }}{\nu !}H_{\nu
}\left( n,\alpha \right) \right) \right\vert _{z=1} \\
&=&n^{-r}e^{-\left( nx+\alpha x^{2}\right) }\left. \left( \frac{d}{dz}%
\right) ^{r}\left[ z^{r}\exp \left( nxz+\alpha \left( xz\right) ^{2}\right) %
\right] \right\vert _{z=1},
\end{eqnarray*}%
where the last equality follows from $\left( \ref{GF}\right) $. Application
of the Leibniz rule for differentiation leads to 
\begin{equation*}
\left( \bar{G}_{n,\alpha }e_{r}\right) \left( x\right) =\frac{r!}{n^{r}}%
e^{-\left( nx+\alpha x^{2}\right) }\sum_{j=0}^{r}\binom{r}{j}\frac{1}{j!}%
\left. \left( \frac{d}{dz}\right) ^{j}\exp \left( nxz+\alpha \left(
xz\right) ^{2}\right) \right\vert _{z=1}.
\end{equation*}%
It is an easy consequence of Faa di Bruno's rule for the differentiation of
composite functions that 
\begin{equation*}
\left( \frac{d}{dz}\right) ^{j}g\left( nxz+\alpha \left( xz\right)
^{2}\right) =\sum_{j/2\leq i\leq j}g^{\left( i\right) }\left( nxz+\alpha
\left( xz\right) ^{2}\right) \frac{j!}{\left( 2i-j\right) !\left( j-i\right)
!}\left( nx+2\alpha x^{2}z\right) ^{2i-j}\left( 2\alpha x^{2}\right) ^{j-i}.
\end{equation*}%
Hence, 
\begin{eqnarray*}
\left( \bar{G}_{n,\alpha }e_{r}\right) \left( x\right) &=&\frac{r!}{n^{r}}%
\sum_{j=0}^{r}\binom{r}{j}\sum_{j/2\leq i\leq j}\frac{1}{\left( 2i-j\right)
!\left( j-i\right) !}\left( nx+2\alpha x^{2}\right) ^{2i-j}\left( 2\alpha
x^{2}\right) ^{j-i} \\
&=&\frac{r!}{n^{r}}\sum_{i=0}^{r}\sum_{j=i}^{2i}\binom{r}{j}\frac{1}{\left(
2i-j\right) !\left( j-i\right) !}\left( nx+2\alpha x^{2}\right)
^{2i-j}\left( 2\alpha x^{2}\right) ^{j-i}
\end{eqnarray*}%
The index transform $j\rightarrow 2i-j$ yields 
\begin{eqnarray*}
\left( \bar{G}_{n,\alpha }e_{r}\right) \left( x\right) &=&\frac{r!}{n^{r}}%
\sum_{i=0}^{r}\sum_{j=0}^{i}\binom{r}{2i-j}\frac{1}{j!\left( i-j\right) !}%
\left( nx+2\alpha x^{2}\right) ^{j}\left( 2\alpha x^{2}\right) ^{i-j} \\
&=&\frac{r!}{n^{r}}\sum_{j=0}^{r}\left( nx+2\alpha x^{2}\right)
^{j}\sum_{i=j}^{r}\binom{r}{2i-j}\frac{1}{j!\left( i-j\right) !}\left(
2\alpha x^{2}\right) ^{i-j}.
\end{eqnarray*}%
Writing 
\begin{eqnarray*}
\left( \bar{G}_{n,\alpha }e_{r}\right) \left( x\right) &=&\frac{r!}{n^{r}}%
\sum_{j=0}^{r}\left( nx+2\alpha x^{2}\right) ^{r-j}\sum_{i=r-j}^{r}\binom{r}{%
2i-r+j}\frac{1}{\left( r-j\right) !\left( i-r+j\right) !}\left( 2\alpha
x^{2}\right) ^{i-r+j} \\
&=&\frac{r!}{n^{r}}\sum_{j=0}^{r}\left( nx+2\alpha x^{2}\right)
^{r-j}\sum_{i=0}^{j}\binom{r}{r+j-2i}\frac{1}{\left( r-j\right) !\left(
j-i\right) !}\left( 2\alpha x^{2}\right) ^{j-i},
\end{eqnarray*}%
application of the binomial formula $\left( nx+2\alpha x^{2}\right)
^{r-j}=\sum_{\mu =0}^{r-j}\binom{r-j}{\mu }\left( nx\right) ^{r-j-\mu
}\left( 2\alpha x^{2}\right) ^{\mu }$ and collecting all terms with $j+\mu
=k $ leads to 
\begin{eqnarray*}
\left( \bar{G}_{n,\alpha }e_{r}\right) \left( x\right)
&=&r!\sum_{k=0}^{r}n^{-k}\sum_{j=0}^{k}\binom{r-j}{k-j}x^{r-k}\sum_{i=0}^{j}%
\binom{r}{r+j-2i}\frac{1}{\left( r-j\right) !\left( j-i\right) !}\left(
2\alpha x^{2}\right) ^{k-i} \\
&=&r!\sum_{k=0}^{r}n^{-k}\sum_{j=0}^{k}\binom{r-j}{k-j}\sum_{i=0}^{j}\binom{r%
}{r+j-2i}\frac{1}{\left( r-j\right) !\left( j-i\right) !}\left( 2\alpha
\right) ^{k-i}x^{r+k-2i} \\
&=&\sum_{k=0}^{r}\frac{r!}{n^{k}\left( r-k\right) !}\sum_{i=0}^{k}\left(
2\alpha \right) ^{k-i}x^{r+k-2i}\sum_{j=i}^{k}\binom{r}{r+j-2i}\frac{1}{%
\left( k-j\right) !\left( j-i\right) !} \\
&=&\sum_{k=0}^{r}\frac{r!}{n^{k}\left( r-k\right) !}\sum_{i=0}^{k}\left(
2\alpha \right) ^{k-i}x^{r+k-2i}\frac{1}{\left( k-i\right) !}\sum_{j=0}^{i}%
\binom{k-i}{j}\binom{r}{i-j}.
\end{eqnarray*}%
By Vandermonde convolution, 
\begin{equation*}
\sum_{j=0}^{i}\binom{k-i}{j}\binom{r}{i-j}=\binom{r+k-i}{i}.
\end{equation*}%
Hence, 
\begin{eqnarray*}
\left( \bar{G}_{n,\alpha }e_{r}\right) \left( x\right) &=&\sum_{k=0}^{r}%
\frac{r!}{n^{k}\left( r-k\right) !}\sum_{i=0}^{k}\left( 2\alpha \right)
^{k-i}x^{r+k-2i}\frac{1}{\left( k-i\right) !}\binom{r+k-i}{i} \\
\frac{1}{k!}\binom{k}{i}\left( x^{2k-i}e_{r}^{\left( k\right) }\left(
x\right) \right) ^{\left( i\right) } &=&\frac{r!}{\left( r-k\right) !}\frac{1%
}{\left( k-i\right) !}\binom{r+k-i}{i}x^{r+k-2i} \\
\left( \bar{G}_{n,\alpha }e_{r}\right) \left( x\right) &=&\sum_{k=0}^{r}%
\frac{1}{k!n^{k}}\sum_{i=0}^{k}\binom{k}{i}\left( x^{2k-i}e_{r}^{\left(
k\right) }\left( x\right) \right) ^{\left( i\right) }\left( 2\alpha \right)
^{k-i} \\
&=&\sum_{k=0}^{r}\frac{1}{n^{k}}c_{k,\alpha }\left( e_{r},x\right) ,
\end{eqnarray*}%
where 
\begin{equation*}
c_{k,\alpha }\left( e_{r},x\right) =\frac{1}{k!}\sum_{i=0}^{k}\binom{k}{i}%
\left( x^{2k-i}e_{r}^{\left( k\right) }\left( x\right) \right) ^{\left(
i\right) }\left( 2\alpha \right) ^{k-i}.
\end{equation*}%
Note that $e_{r}^{\left( k\right) }\left( x\right) =0$ and, therefore, $%
c_{k,\alpha }\left( e_{r},x\right) =0$, for all integers $k>r$.
\end{proof}

\begin{lemma}
\label{lemma-central-moments}For $s=0,1,2,\ldots $, the central moments of
the operators $\left( \ref{def-G-n-alpha-Durrmeyer}\right) $ have the explicit
representation 
\begin{equation*}
\left( \bar{G}_{n,\alpha }\psi _{x}^{s}\right) \left( x\right)
=\sum_{k=\left\lfloor \left( s+1\right) /2\right\rfloor }^{s}\frac{1}{k!n^{k}%
}\binom{k}{s-k}\sum_{i=0}^{2k-s}\binom{2k-s}{i}\frac{s!\left( k+i\right) !}{%
\left( s-k+2i\right) !}x^{s-k+2i}\left( 2\alpha \right) ^{i}.
\end{equation*}
\end{lemma}

\begin{proof}
Our starting-point is the obvious relation 
\begin{equation*}
\left( \bar{G}_{n,\alpha }\psi _{x}^{s}\right) \left( x\right)
=\sum_{r=0}^{s}\binom{s}{r}\left( -x\right) ^{s-r}\left( \bar{G}_{n,\alpha
}e_{r}\right) \left( x\right)
\end{equation*}%
for the central moments. By Lemma~\ref{lemma-moments}, we have\ 
\begin{eqnarray*}
\left( \bar{G}_{n,\alpha }\psi _{x}^{s}\right) \left( x\right)
&=&\sum_{k=0}^{s}\frac{1}{k!n^{k}}\sum_{r=0}^{s}\binom{s}{r}\left( -x\right)
^{s-r}\sum_{i=0}^{k}\binom{k}{i}\left( 2\alpha \right) ^{k-i}\left. \left( 
\frac{d}{dt}\right) ^{i}\left( t^{2k-i}\left( \frac{d}{dt}\right)
^{k}t^{r}\right) \right\vert _{t=x} \\
&=&\sum_{k=0}^{s}\frac{1}{k!n^{k}}\sum_{i=0}^{k}\binom{k}{i}\left( 2\alpha
\right) ^{k-i}\left. \left( \frac{d}{dt}\right) ^{i}\left( t^{2k-i}\left( 
\frac{d}{dt}\right) ^{k}\left( t-x\right) ^{s}\right) \right\vert _{t=x}
\end{eqnarray*}%
Now, 
\begin{eqnarray*}
\left. \left( \frac{d}{dt}\right) ^{i}\left( t^{2k-i}\left( \frac{d}{dt}%
\right) ^{k}\left( t-x\right) ^{s}\right) \right\vert _{t=x} &=&\frac{s!}{%
\left( s-k\right) !}\left. \left( \frac{d}{dt}\right) ^{i}\left(
t^{2k-i}\left( t-x\right) ^{s-k}\right) \right\vert _{t=x} \\
&=&\frac{s!}{\left( s-k\right) !}\binom{i}{s-k}\left[ \left( \frac{d}{dx}%
\right) ^{i-s+k}x^{2k-i}\right] \left( s-k\right) ! \\
&=&s!\binom{i}{s-k}\frac{\left( 2k-i\right) !}{\left( k-2i+s\right) !}%
x^{k-2i+s},
\end{eqnarray*}%
such that 
\begin{equation*}
\left( \bar{G}_{n,\alpha }\psi _{x}^{s}\right) \left( x\right)
=\sum_{k=0}^{s}\frac{1}{k!n^{k}}\sum_{i=s-k}^{k}\binom{k}{i}\binom{i}{s-k}%
\frac{s!\left( 2k-i\right) !}{\left( k-2i+s\right) !}x^{k-2i+s}\left(
2\alpha \right) ^{k-i}
\end{equation*}%
The binomial identity 
\begin{equation*}
\binom{k}{i}\binom{i}{s-k}=\binom{k}{s-k}\binom{2k-s}{k-i}
\end{equation*}%
yields 
\begin{eqnarray*}
\left( \bar{G}_{n,\alpha }\psi _{x}^{s}\right) \left( x\right)
&=&\sum_{k=0}^{s}\frac{1}{k!n^{k}}\binom{k}{s-k}\sum_{i=s-k}^{k}\binom{2k-s}{%
k-i}\frac{s!\left( 2k-i\right) !}{\left( k-2i+s\right) !}x^{k-2i+s}\left(
2\alpha \right) ^{k-i} \\
&=&\sum_{k=0}^{s}\frac{1}{k!n^{k}}\binom{k}{s-k}\sum_{i=0}^{2k-s}\binom{2k-s%
}{i}\frac{s!\left( k+i\right) !}{\left( s-k+2i\right) !}x^{s-k+2i}\left(
2\alpha \right) ^{i}.
\end{eqnarray*}%
Since $\binom{k}{s-k}=0$ if $s>2k$, we have $\left( \bar{G}_{n,\alpha }\psi
_{x}^{s}\right) \left( x\right) =O\left( n^{-\left\lfloor \left( s+1\right)
/2\right\rfloor }\right) $ as $n\rightarrow \infty $. This completes the
proof of Lemma~\ref{lemma-central-moments}.
\end{proof}

In order to derive Theorem~\ref{theorem-expansion}, a general approximation
theorem due to Sikkema \cite[Theorem~3]{Sikkema1970a} will be applied. To
this end, we need some notation. Let $I$ be a real interval and $x\in I$.
For $s\in \mathbb{N}$, let $H^{\left( s\right) }\left( x\right) $ denote the
class of all locally bounded real functions $f:I\rightarrow \mathbb{R}$,
which are $s$ times differentiable at $x$. In the case that $I$ is an
infinite interval, $f$ has to satisfy the additional condition $f\left(
t\right) =O\left( \left\vert t\right\vert ^{s}\right) $ as $t\rightarrow
+\infty $ or $t\rightarrow -\infty $, respectively, or both if $I=\mathbb{R}$%
. An inspection of the proof of Sikkema's result reveals that it can be
stated in the following form which is more appropriate for our purposes.

\begin{lemma}
\label{Lemma-Sikkema}Let $q\in \mathbb{N}$ and let $\left( L_{n}\right)
_{n\in \mathbb{N}}$ be a sequence of positive linear operators, $%
L_{n}:H^{\left( 2q\right) }\left( x\right) \rightarrow C\left[ a,b\right] $, 
$x\in \left[ a,b\right] $. Suppose that the operators $L_{n}$ apply to $\psi
_{x}^{2q+1}$ and to $\psi _{x}^{2q+2}$. Then the condition 
\begin{equation*}
\left( L_{n}\psi _{x}^{s}\right) \left( x\right) =O\left( n^{-\left \lfloor
\left( s+1\right) /2\right \rfloor }\right) \text{ \qquad }\left(
n\rightarrow \infty \right) ,\text{ \qquad for }s=0,1,\ldots ,2q+2,
\end{equation*}%
implies, for each function $f\in H^{\left( 2q\right) }\left( x\right) $, the
asymptotic relation 
\begin{equation*}
\left( L_{n}f\right) \left( x\right) =\sum_{s=0}^{2q}\frac{f^{\left(
s\right) }\left( x\right) }{s!}\left( L_{n}\psi _{x}^{s}\right) \left(
x\right) +o\left( n^{-q}\right) \text{ }\qquad \left( n\rightarrow \infty
\right) .
\end{equation*}
\end{lemma}

Now we are in position to prove our main result.

\begin{proof}[Proof of Theorem~\protect\ref{theorem-expansion}]
Application of Sikkema's theorem (Lemma~\ref{Lemma-Sikkema}) yields 
\begin{equation*}
\left( \bar{G}_{n,\alpha }f\right) \left( x\right) =\sum_{s=0}^{2q}\frac{%
f^{\left( s\right) }\left( x\right) }{s!}\left( \bar{G}_{n,\alpha }\psi
_{x}^{s}\right) \left( x\right) +o\left( n^{-q}\right) \text{ }\qquad \left(
n\rightarrow \infty \right) .
\end{equation*}%
By Lemma~\ref{lemma-moments}, we infer that 
\begin{eqnarray*}
&&\sum_{s=0}^{2q}\frac{f^{\left( s\right) }\left( x\right) }{s!}\left( \bar{G%
}_{n,\alpha }\psi _{x}^{s}\right) \left( x\right)  \\
&=&\sum_{k=0}^{q}\frac{1}{k!n^{k}}\sum_{s=k}^{2k}f^{\left( s\right) }\left(
x\right) \binom{k}{s-k}\sum_{i=0}^{2k-s}\binom{2k-s}{i}\frac{\left(
k+i\right) !}{\left( s-k+2i\right) !}x^{s-k+2i}\left( 2\alpha \right)
^{i}+o\left( n^{-q}\right)  \\
&=&\sum_{k=0}^{q}\frac{1}{k!n^{k}}\sum_{s=0}^{k}f^{\left( k+s\right) }\left(
x\right) \binom{k}{s}\sum_{i=0}^{k-s}\binom{k-s}{i}\frac{\left( k+i\right) !%
}{\left( s+2i\right) !}x^{s+2i}\left( 2\alpha \right) ^{i}+o\left(
n^{-q}\right) 
\end{eqnarray*}%
as $n\rightarrow \infty $. Using 
\begin{equation*}
\binom{k}{s}\binom{k-s}{i}=\binom{k}{i}\binom{k-i}{s}
\end{equation*}%
we obtain 
\begin{eqnarray*}
&&\left( \bar{G}_{n,\alpha }f\right) \left( x\right)  \\
&=&\sum_{k=0}^{q}\frac{1}{k!n^{k}}\sum_{i=0}^{k}\binom{k}{i}\left( 2\alpha
\right) ^{i}\sum_{s=0}^{k-i}\binom{k-i}{s}f^{\left( k+s\right) }\left(
x\right) \frac{\left( k+i\right) !}{\left( s+2i\right) !}x^{s+2i}+o\left(
n^{-q}\right)  \\
&=&\sum_{k=0}^{q}\frac{1}{k!n^{k}}\sum_{i=0}^{k}\binom{k}{i}\left( 2\alpha
\right) ^{k-i}\sum_{s=0}^{i}\binom{i}{s}f^{\left( k+s\right) }\left(
x\right) \frac{\left( 2k-i\right) !}{\left( s+2k-2i\right) !}%
x^{s+2k-2i}+o\left( n^{-q}\right) 
\end{eqnarray*}%
as $n\rightarrow \infty $. Observing that $\left( x^{2k-i}\right) ^{\left(
i-s\right) }=\frac{\left( 2k-i\right) !}{\left( s+2k-2i\right) !}x^{s+2k-2i}$
it follows, by Leibniz rule for differentiation, that 
\begin{equation*}
\left( \bar{G}_{n,\alpha }f\right) \left( x\right) =\sum_{k=0}^{q}\frac{1}{%
k!n^{k}}\sum_{i=0}^{k}\binom{k}{i}\left( x^{2k-i}f^{\left( k\right) }\left(
x\right) \right) ^{\left( i\right) }\left( 2\alpha \right) ^{k-i}+o\left(
n^{-q}\right) 
\end{equation*}%
as $n\rightarrow \infty $. 
\end{proof}

\bigskip


\end{document}